\documentclass{article}
\usepackage{amsmath}
\usepackage{amssymb, amscd, amsthm, amsfonts, latexsym}
\usepackage{enumerate}
\usepackage[dvips]{graphicx,color}

\usepackage{latexsym}
\newcommand{\no}{\noindent}
\newtheorem{thm}{Theorem}[section]
\newtheorem{prop}[thm]{Proposition}
\newtheorem{cor}[thm]{Corollary}
\newtheorem{rem}[thm]{Remark}
\newtheorem{defi}[thm]{Definition}
\newtheorem{ex}[thm]{Example}
\newtheorem{lem}[thm]{Lemma}

\numberwithin{equation}{section}

\newcommand{\R}{\mathbb{R}}

\newcommand{\ds}{\displaystyle}

\newcommand{\sm}{\setminus}
\newcommand{\pd}{\partial}
\newcommand{\al}{\alpha}

\newcommand{\ga}{\gamma}
\newcommand{\Ga}{\Gamma}
\newcommand{\de}{\delta}
\newcommand{\De}{\Delta}
\newcommand{\ep}{\varepsilon}
\newcommand{\la}{\lambda}
\newcommand{\f}{\varphi}

\newcommand{\Ome}{\Omega}

\renewcommand{\(}{\left(}
\renewcommand{\)}{\right)}
\renewcommand{\lvert}{\left\vert}
\renewcommand{\rvert}{\right\vert}

\DeclareMathOperator{\diam}{diam}
\DeclareMathOperator{\conv}{conv}
\DeclareMathOperator{\supp}{supp}

\DeclareMathOperator{\dist}{dist}
\DeclareMathOperator{\Vol}{Vol}

\DeclareMathOperator{\Refl}{Refl}

\DeclareMathOperator{\sign}{sign}
\DeclareMathOperator{\Ker}{Ker}

\def\c#1{\overset{\mbox{\tiny $\circ$}}{#1}} 

\begin{document}

\title{\bf Stationary radial centers and characterization of convex polyhedrons}

\author{\bf Shigehiro Sakata}

\date{\today}

\maketitle

\begin{abstract} 
We investigate centers of a body (the closure of a bounded open set) in $\R^m$ defined as maximum points of potentials. In particular, we study centers defined by the Riesz potential and by Poisson's integral. These centers, in general, depend on parameters and move with respect to the parameters. We give a necessary and sufficient condition for the existence of a center independent of a parameter.\\

\no{\it Keywords and phrases}.
Radial center, $r^{\al -m}$-center, illuminating center, hot spot, power concavity, balance law.\\
\no 2010 {\it Mathematics Subject Classification}:
35B38, 52B15, 26B25, 51M16, 52A10.
\end{abstract}
\section{Introduction}

Let $\Ome$ be a body (the closure of a bounded open set) in $\R^m$. In this paper, we investigate centers of $\Ome$ from three motivations.

The first motivation of this paper comes from Moszy\'{n}ska's {\it radial center} of a star body $A$. In \cite{M1}, Moszy\'{n}ska looked for a ``good'' position of the origin for $A$ and defined a center of $A$ as a maximizer of a function of the form
\begin{equation}\label{Mos}
\Phi_A (x) = \int_{S^{m-1}} \phi \( \rho_{A-x} (v) \) d\sigma (v) ,\ 
x \in \Ker A := \left\{ x \in A \lvert \forall y \in A,\ \overline{xy} \subset A \right\} \right. .
\end{equation}
Here, $\rho_{A-x} (v) = \max \{ \la \geq 0 \vert \la v +x \in A \}$ is the {\it radial function} of $A$ with respect to $x$. A maximum or minimum point of $\Phi_A$ is called a {\it radial center of $A$ associated with $\phi$}. In particular, when $\phi (\rho ) = \rho^\al /\al$ $(\al >0)$, it is called a {\it radial center of order $\al$}. We refer to \cite[Introduction]{M1} for the historical background of the study on radial centers (see also \cite[Part III]{M2}).

In \cite{O1}, O'Hara investigated a maximizer of the potential 
\begin{equation}
V_\Ome^{(\al )}=
\begin{cases}
\ds \sign (m-\al ) \int_\Ome \lvert x-y \rvert^{\al -m} dy &(0< \al \neq m),\\
\ds -\int_\Ome \log \lvert x-y \rvert dy &(\al =m) ,
\end{cases}
,\ x \in \R^m.
\end{equation} 
Since the potential $V_\Ome^{(\al )}$ is defined for every (not necessarily star shaped) body $\Ome$, we can regard a maximum point of $V_\Ome^{(\al )}$ as a radial center of $\Ome$ of order $\al$. Moreover, taking the Hadamard finite part of $V_\Ome^{(\al )}(x)$ for every interior point $x$ of $\Ome$, he defined a radial center of $\Ome$ for every $\al \in \R$ and called a maximum point of $V_\Ome^{(\al )}$ an {\it $r^{\al-m}$-center}.

In \cite{O1}, it was shown that if $m \geq 2$ and $\al \geq m+2$, then $\Ome$ has a unique $r^{\al -m}$-center, and that, as $\al$ goes to infinity, the unique $r^{\al-m}$-center converges to the {\it circumcenter} (the center of the minimal ball containing $\Ome$) of $\Ome$.  It is easy to check that the $r^2$-center is the centroid (center of mass) of $\Ome$. Thus an $r^{\al -m}$-center moves with respect to $\al$ in general.

The second motivation of this paper comes from the problem ``Where should we put a streetlight in a triangular park?'' proposed by PISA (Programme for International Student Assessment) in 2003. PISA's answer for this problem was the circumcenter of a triangular park. In other words, PISA considered this problem as a min-max problem: Minimize the distance between a streetlight and every vertex of the park (the darkest point). 

In \cite{Sh}, Shibata studied PISA's problem in view of potential theory and suggested another point to put a streetlight. He defined the ``brightness'' of a triangular park $\triangle$ with a streetlight of height $h$ located at point $x$ as the function
\begin{equation}\label{brightness}
\R^2 \ni x \mapsto \int_\triangle \frac{h}{\( \lvert x-y \rvert^2+ h^2 \)^{3/2}} dy \in \R .
\end{equation}
A maximum point of the brightness function is called an {\it illuminating center} of $\triangle$ of height $h$. The definition of \eqref{brightness} is based on Kepler's inverse square law. Precisely, for a point $y \in \triangle$, we define the ``local brightness'' at $y$ as 
\begin{equation}
\R^2 \ni x \mapsto \( \sqrt{ \lvert x-y \rvert^2 +h^2} \)^{-2} \frac{h}{\sqrt{ \lvert x-y \rvert^2 +h^2}} \in \R ,
\end{equation}
and we integrate the local brightness over $\triangle$.

In \cite{Skt1}, the author generalized Shibata's notion for a body $\Ome$ in $\R^m$ and indicated that the brightness function is proportional to {\it Poisson's integral},
\begin{equation}
P_\Ome (x, h) = \frac{2h}{\sigma_m \( S^m \)} \int_\Ome \frac{1}{\( \lvert x-y \rvert^2 + h^2 \)^{(m+1)/2}} dy ,\ x \in \R^m,\ h>0 .
\end{equation}
Here, we remark that, putting
\begin{equation}
\phi (\rho ) = \int_0^\rho \frac{r^{m-1}}{\( r^2 +h^2 \)^{(m+1)/2}} dr
\end{equation}
in \eqref{Mos}, an illuminating center is an example of a radial center.

In \cite{Skt1}, it was shown that if $h \geq \sqrt{m+2} \diam \Ome$, then $\Ome$ has a unique illuminating center, and that, as $h$ goes to infinity, the unique illuminating center converges to the centroid of $\Ome$. In \cite{Skt3}, it was shown that the limit point of any convergent sequence of illuminating centers $c(h_j)$ of height $h_j$ with $h_j \to 0^+$ is an $r^{-1-m}$-center. Thus an illuminating center moves with respect to $h$ in general. 

The third motivation of this paper comes from the study on geometric properties for the heat equation. It is well-known that the function
\begin{equation}
W_\Ome (x,t) = \frac{1}{\( 4\pi t\)^{m/2}} \int_\Ome \exp \( -\frac{\lvert x-y \rvert^2}{4t} \) dy,\ x\in \R^m ,\ t>0,
\end{equation}
is the solution of the Cauchy problem for the heat equation with initial datum $\chi_\Ome$. A maximum point of $W_\Ome (\cdot ,t)$ is called a {\it hot spot} of $\Ome$ at time $t$. We remark that, putting
\begin{equation}
\phi (\rho ) = \int_0^\rho \exp \( -\frac{r^2}{4t} \) r^{m-1} dr
\end{equation}
in \eqref{Mos}, a hot spot is an example of a radial center.

In \cite{JS}, it was indicated that if $t \geq ( \diam \Ome )^2/2$, then $\Ome$ has a unique hot spot. In \cite{CK}, it was shown that, as $t$ goes to infinity, the unique hot spot converges to the centroid of $\Ome$. In \cite{KP}, it was shown that the limit point of any convergent sequence of hot spots $h(t_j )$ at time $t_j$ with $t_j \to 0^+ $ is an {\it incenter} (the center of a maximal ball contained in $\Ome$) of $\Ome$. Thus a hot spot moves with respect to $t$ in general. 

In \cite{MS1}, Magnanini and Sakaguchi gave a necessary and sufficient condition for the existence of a critical point of $W_\Ome(\cdot ,t)$ not moving with respect to $t$. The condition is called the {\it balance law}. Precisely, the origin is a stationary critical point of $W_\Ome (\cdot t,)$ if and only if $\Ome$ satisfies the equation
\begin{equation}
\int_{rS^{m-1} \cap \Ome} v d \sigma (v) =0
\end{equation}
for any $r \geq 0$. Using the balance law, in \cite{MS4}, they characterized a convex polyhedron having a stationary hot spot. (To tell the truth, in \cite{MS4}, they studied a stationary hot spot for the Dirichlet problem which is more difficult than the Cauchy problem.)

Among the above three motivations, there is a common feature. They are the studies on a critical point (or a maximum point) of a function depending on a parameter. Moreover, their critical points move with respect to parameters in general. From such a background, in this paper, we investigate a necessary and sufficient condition for the existence of a critical point of $V_\Ome^{(\al )}$ and $P_\Ome (\cdot ,h)$ not moving with respect to $\al$ and $h$, respectively. We show that the condition is Magnanini and Sakaguchi's balance law. Furthermore, using the balance law and a stationary $r^{\al -m}$-center, we give an elementary proof for the characterization of a convex polygon shown in \cite[Theorem 6]{MS4}. Precisely, in an elementary argument, we show that a triangle and a convex quadrangle having a stationary $r^{\al -2}$-center are a equilateral triangle and a parallelogram, respectively. \\

Throughout this paper, $\c{X}$ (or $X^\circ$), $\bar{X}$, $X^c$ and $\conv X$ denote the interior, closure, complement and convex hull of a set $X$ in $\R^m$, respectively. We denote the surface Lebesgue measure of an $N$-dimensional space by $\sigma_N$. In particular, the symbol $\sigma$ is used in the case of $N=m-1$ for short. For a number $a$, a point $p \in \R^m$ and a set $X \subset \R^m$, we use the notation (the Minkowski addition) $aX+p = \{ ax+p \vert x \in X \}$. In particular, if necessary, we denote by $B_\rho (x) =\rho B^m +x$ and $S_\rho (x) = \rho S^{m-1}+x$ the $m$-dimensional closed ball of radius $\rho$ centered at $x$ and the $(m-1)$-dimensional sphere of radius $\rho$ centered at $x$, respectively.  \\

\no{\bf Acknowledgements.} The author would like to express his deep gratitude to Professors Jun O'Hara, Rolando Magnanini and Paolo Salani. Professor O'Hara suggested  the study on the uniqueness of an illuminating center of a convex body. Professors Magnanini and Salani gave him many kind advices on the power concavity of functions and the balance law.

The author is partially supported by Waseda University Grant for Special Research Project 2015K-335 and JSPS Kakenhi Grant Number 26887041.
\section{Preliminaries}
\subsection{Riesz potentials of a body}

In this subsection, we prepare terminologies on the Riesz potential (and its extension) from \cite{O1}. In particular, we remark the differentiability of Riesz potentials.

\begin{defi}[{\cite[Definition 2.1]{O1}}]\label{O1def2.1}
{\rm 
Let $\Ome$ be a body (the closure of a bounded open set) in $\R^m$. We define the {\it $r^{\al-m}$-potential} of $\Ome$ of order $\al$ as follows:
\begin{enumerate}[(1)]
\item When $0< \al \neq m$, we define
\[
V_\Ome^{(\al )} (x) = \sign (m-\al ) \int_\Ome \lvert x-y \rvert^{\al -m} dy,\ x \in \R^m.
\]
\item When $\al =m$, we define
\[
V_\Ome^{(m)} (x) = -\int_\Ome \log \lvert x-y \rvert dy ,\ x \in \R^m .
\]
\item When $\al =0$, we define
\[
V_\Ome^{(0)}(x)=
\begin{cases}
\ds \lim_{\ep \to 0^+} \( \int_{\Ome \sm B_\ep (x)} \lvert x-y \rvert^{-m} dy -\sigma \( S^{m-1} \) \log \frac{1}{\ep} \) &\( x \in \c{\Ome} \) ,\\
\ds \int_\Ome \lvert x-y \rvert^{-m} dy & \( x \in \Ome^c \) .
\end{cases}
\]
\item When $\al <0$, we define
\[
V_\Ome^{(\al )} (x) =
\begin{cases}
\ds \lim_{\ep \to 0^+} \( \int_{\Ome \sm B_\ep (x)} \lvert x-y \rvert^{\al -m} dy -\frac{\sigma \( S^{m-1} \)}{-\al} \frac{1}{\ep^{-\al}} \) &\( x \in \c{\Ome} \) ,\\
\ds \int_\Ome \lvert x-y \rvert^{\al -m} dy & \( x \in \Ome^c \) .
\end{cases}
\]
\end{enumerate}
When $0<\al <m$, $V_\Ome^{(\al)}$ is so-called the {\it Riesz potential}.
}
\end{defi}

\begin{prop}[{\cite[Proposition 2.5]{O1}}]\label{O1prop2.5}
Let $\Ome$ be a body in $\R^m$. 
\begin{enumerate}[$(1)$]
\item When $\al =0$ and $x$ is in the interior of $\Ome$, for any $0< \ep < \dist ( x, \Ome^c )$, we have
\[
V_\Ome^{(0)} (x) =\int_{\Ome \sm B_\ep (x)} \lvert x-y \rvert^{-m} dy -\sigma \( S^{m-1} \) \log \frac{1}{\ep}.
\]
\item When $\al <0$ and $x$ is in the interior of $\Ome$, we have
\[
V_\Ome^{(\al )} (x) = -\int_{\Ome^c} \lvert x-y \rvert^{\al -m} dy .
\]
\end{enumerate}
\end{prop}

\begin{prop}[{\cite[Proposition 2.9 and p.379]{O1}}]\label{O1prop2.9}
Let $\Ome$ be a body in $\R^m$ with a piecewise $C^1$ boundary. We denote by $n$ the unit outer normal vector filed of $\pd \Ome$. If $x$ is not on the boundary of $\Ome$, then we have
\[
\nabla V_\Ome^{(\al )} (x) =
\begin{cases}
\ds -\sign (m-\al ) \int_{\pd \Ome} \lvert x-y \rvert^{\al -m} n(y) d\sigma (y) &(\al \neq m) ,\\
\ds \int_{\pd \Ome} \log \lvert x-y \rvert n(y) d\sigma (y) &( \al =m) .
\end{cases}
\]
\end{prop}

\begin{lem}\label{diffV01}
Let $\Ome$ be a body in $\R^m$ with a piecewise $C^1$ boundary. We denote by $n$ the unit outer normal vector filed of $\pd \Ome$. If $x$ is in the interior of $\Ome$, then we have
\[
\nabla V_\Ome^{(\al )} (x)=
\begin{cases}
\ds (\al -m) \sign (m-\al ) \int_{\Ome \sm B_\ep (x)} \lvert x-y \rvert^{\al -m-2} (x-y ) dy &( \al \neq m),\\
\ds - \int_{\Ome \sm B_\ep (x)} \frac{x-y}{\lvert x-y \rvert^2} dy &(\al =m) 
\end{cases} 
\]
for any $0< \ep < \dist (x,\Ome^c )$.
\end{lem}

\begin{proof}
When $m \geq 2$ and $\al \notin [0,1]$, Definition \ref{O1def3.1} and Proposition \ref{O1prop2.5} guarantee the statement (see also Remark \ref{diffV}). Let us give a proof for the case of $m \geq 2$ and $0\leq \al \leq 1$. The one-dimensional case goes parallel.

Fix an interior point $x$ of $\Ome$. We take an $0< \ep < \dist (x, \Ome^c)$. Thanks to Proposition \ref{O1prop2.9}, we have
\begin{align*}
\nabla V_\Ome^{(\al )}(x)
&= - \( \int_{\pd \Ome} - \int_{\pd B_\ep (x)} \) \lvert x-y \rvert^{\al -m} n(y) d \sigma (y) \\
&=- \int_{\pd \( \Ome \sm B_\ep (x) \)} \lvert x-y \rvert^{\al -m} n(y) d \sigma (y) \\
&=  \int_{\Ome \sm B_\ep (x)} \nabla \lvert x-y \rvert^{\al -m} dy \\
&= (\al -m) \int_{\Ome \sm B_\ep (x)} \lvert x-y \rvert^{\al -m-2} (x-y) dy .
\end{align*}
Here, the third equality follows from Stokes' theorem. 
\end{proof}

\begin{lem}[{\cite[Lemma 2.13]{O1}}]\label{O1lem2.13}
Let $\Ome$ be a body in $\R^m$.
\begin{enumerate}[$(1)$]
\item For a point $z$ on the boundary of $\Ome$, we assume the existence of positive constants $\ep$ and $\theta$ such that $B_\ep (z) \cap \Ome^c$ contains an open cone of vertex $z$ and aperture angle $\theta$. If $\al \leq 0$, then the function $V_\Ome^{(\al )}(x)$ goes to $-\infty$ as $x \in \c{\Ome}$ approaches to $z$.
\item For a point $z$ on the boundary of $\Ome$, we assume the existence of positive constants $\ep$ and $\theta$ such that $B_\ep (z) \cap \Ome$ contains a closed cone of vertex $z$ and aperture angle $\theta$. If $\al \leq 0$, then the function $V_\Ome^{(\al )}(x)$ goes to $+\infty$ as $x \in \Ome^c$ approaches to $z$.
\end{enumerate}
\end{lem}


\begin{lem}\label{diffVbd}
Let $\Ome$ be a body in $\R^m$. A point $z$ on the boundary of $\Ome$ satisfies the following conditions:
\begin{itemize}
\item There is a motion $g \in O( m)$ such that $g(\Ome -z)$ around the origin can be obtained as the space under the graph of a continuous function $f :\R^{m-1} \to \R$. Namely, there are positive constants $\rho$ and $b$ with 
\[
\( \rho B^{m-1} \times [-b, b] \) \cap g \( \Ome -z \) 
= \left\{ y = \( y',y_m \) \in \R^{m-1} \times \R \lvert -b \leq y_m \leq f \( y' \) ,\ y' \in \rho B^{m-1} \right\} \right. .
\]
\item The boundary of $\Ome$ satisfies the inner and outer cone conditions at $z$. Namely, there is a positive constant $C$ such that, for any $y' \in \rho B^{m-1}$, the inequality $\vert f ( y' ) \vert \leq C \vert y' \vert$ holds.
\end{itemize}
If $0 < \al \leq 1$, then we have
\[
\lim_{h \to 0} \frac{V_{g ( \Ome -z)}^{(\al)} \( he_m \) -V_{g(\Ome -z )}^{(\al )}(0)}{h} = -\infty ,
\]
that is, $V_\Ome^{(\al)}$ is not differentiable at $z$.
\end{lem}

\begin{proof}
Let us consider the case of $m \geq 2$. The one-dimensional case is directly checked.

After a motion of $\R^m$, we may assume that $z=0$ and 
\[
\Ome' 
:= \( \rho B^{m-1} \times [-b,b] \) \cap \Ome 
= \left\{ y = \( y',y_m \) \in \R^{m-1} \times \R \lvert -b \leq y_m \leq f \( y' \) ,\ y' \in \rho B^{m-1} \right\} \right. .
\]
We decompose $V_\Ome^{(\al)} = V_{\Ome'}^{(\al )} +V_{\Ome \sm \Ome'}^{(\al)}$. Since $\Ome \sm \Ome'$ does not contain the origin, the function $V_{\Ome \sm \Ome'}^{(\al )}$ is differentiable at the origin. Let us consider the differentiability of $V_{\Ome'}^{(\al )}$ at the origin.

Fix an arbitrary $0<h<b$. We have
\[
\frac{V_{\Ome'}^{(\al)} \( he_m \) - V_{\Ome'}^{(\al )} (0)}{h}
=\frac{1}{h} \int_{\rho B^{m-1}} \( \( \int_{-b-h}^{-b} - \int_{f \( y' \)-h}^{f \( y' \)} \) \lvert y \rvert^{\al -m} dy_m \) dy' .
\]
The continuity of the function $y_m \mapsto \vert y \vert^{\al -m}$ implies
\[
\lim_{h \to 0^+} \frac{1}{h} \int_{\rho B^{m-1}}  \( \int_{-b-h}^{-b}  \lvert y \rvert^{\al -m} dy_m \) dy' 
= \int_{\rho B^{m-1}} \( \lvert y' \rvert^2 + b^2 \)^{(\al -m)/2} dy' .
\]
Since $\al - m \leq 1-m  <0$, the cone conditions of the boundary of $\Ome$ at the origin implies 
\begin{align*}
&\frac{1}{h} \int_{\rho B^{m-1}} \( \int_{f \( y' \)-h}^{f \( y' \)} \lvert y \rvert^{\al -m} dy_m \) dy'\\ 
&\geq \frac{1}{h} \int_{\rho B^{m-1}} \( \int_{f \( y' \)-h}^{f \( y' \)} \( \lvert y' \rvert^2 + \( \lvert f \( y' \) \rvert +h \)^2 \)^{(\al -m)/2} dy_m \) dy' \\ 
&\geq \frac{1}{h} \int_{\rho B^{m-1}} \( \int_{f \( y' \)-h}^{f \( y' \)} \( \lvert y' \rvert^2 + \(  C \lvert y'  \rvert +h \)^2 \)^{(\al -m)/2} dy_m \) dy' \\
&= \sigma_{m-2} \( S^{m-2} \) \( 1+C^2 \)^{(\al -m)/2} \int_0^\rho \( \( r+ \frac{Ch}{1+C^2} \)^2 + \( \frac{h}{1+C^2} \)^2 \)^{(\al -m)/2} r^{m-2} dr .
\end{align*}
Here, for the equality, we used the polar coordinate $y' \mapsto rv$ ($0\leq r \leq \rho$, $v \in S^{m-2}$). Since $(\al -m) + (m-2) \leq -1$, the above right hand side diverges to $+\infty$ as $h$ tends to $0^+$. Hence we obtain
\[
\lim_{h \to 0^+} \frac{V_{\Ome'}^{(\al)} \( he_m \) - V_{\Ome'}^{(\al )} (0)}{h} =-\infty .
\]

In the same manner as above, we can obtain the same conclusion for $-b < h<0$. Thus the proof is completed.
\end{proof}

\begin{rem}\label{diffV}
{\rm Let $\Ome$ be a body in $\R^m$. Let us summarize the differentiability of $V_\Ome^{(\al )}$.
\begin{enumerate}[(1)]
\item If $\al >1$, then we have
\[
\nabla V_\Ome^{(\al )}(x) 
=
\begin{cases}
\ds (\al -m)  \sign (m-\al ) \int_\Ome \lvert x-y \rvert^{\al -m-2} (x-y) dy  &( \al \neq m),\\
\ds -\int_\Ome \frac{x-y}{\lvert x-y \rvert^2} dy &(\al =m)
\end{cases}
\]
for any $x \in \R^m$ (see \cite[Proposition 2.6]{Skt1}).
\item If $\al \leq 1$ and $x \in \Ome^c$, then we have
\[
\nabla V_\Ome^{(\al )}(x) 
= 
\begin{cases}
\ds (\al -m)  \int_\Ome \lvert x-y \rvert^{\al -m-2} (x-y) dy &( \al <m=1 \ {\rm or}\ m \geq 2),\\
\ds -\int_\Ome \frac{x-y}{\lvert x-y \rvert^2} dy &(\al =m=1) .
\end{cases}
\]
This fact can be shown in the same manner as in \cite[Proposition 2.6]{Skt1}.
\item If $\al <0$ and $x \in \c{\Ome}$, then we have
\[
\nabla V_\Ome^{(\al )}(x) 
= - (\al -m)  \int_{\Ome^c} \lvert x-y \rvert^{\al -m-2} (x-y) dy .
\]
Using the second assertion in Proposition \ref{O1prop2.5}, this fact can be shown in the same manner as in \cite[Proposition 2.6]{Skt1}.
\item If $0 \leq \al \leq 1$, $x \in \c{\Ome}$ and $\Ome$ has a piecewise $C^1$ boundary, then $\nabla V_\Ome^{(\al )}(x)$ is given in Lemma \ref{diffV01}.
\item If $\al  \leq 1$ and $x \in \pd \Ome$, then, from Lemmas \ref{O1lem2.13} and \ref{diffVbd}, $\nabla V_\Ome^{(\al )}(x)$ does not exist in general.
\end{enumerate}
}
\end{rem}

\begin{defi}[{\cite[Definitions 3.1 and 3.17]{O1}}]\label{O1def3.1}
{\rm Let $\Ome$ be a body in $\R^m$.
\begin{enumerate}[(1)]
\item When $\al >0$, a point $c$ is called an {\it $r^{\al -m}$-center} of $\Ome$ if it gives the maximum value of $V_\Ome^{(\al)}$.
\item When $\al \leq 0$, a point $c$ is called an {\it $r^{\al -m}$-center} of $\Ome$ if it gives the maximum value of the restriction of $V_\Ome^{(\al)}$ to the interior of $\Ome$.
\item A point $c$ is called an {\it $r^\infty$-center} or a {\it circumcenter} of $\Ome$ if it gives the minimum value of the function
\[
\R^m \ni x \mapsto \max_{y \in \Ome} \lvert x-y \rvert \in \R .
\]
\item A point $c$ is called an {\it $r^{-\infty}$-center} or an {\it incenter} of $\Ome$ if it gives the maximum value of the function 
\[
\R^m \ni x \mapsto \max_{y \in \overline{\Ome^c}} \lvert x-y \rvert \in \R .
\]
\end{enumerate}
}
\end{defi}

\begin{ex}[{\cite[Example 3.1]{O1}}]\label{O1ex3.1}
{\rm 
Let $\Ome$ be a body in $\R^m$. The $r^2$-center of $\Ome$ is the centroid of $\Ome$,
\[
G_\Ome = \frac{1}{\Vol (\Ome)} \int_\Ome ydy .
\]
}
\end{ex}

\begin{thm}[{\cite[Theorem 3.5]{O1}}]\label{O1thm3.5}
Let $\Ome$ be a body in $\R^m$. For any $\al \in \R$, $\Ome$ has an $r^{\al -m}$-center.
\end{thm}
\subsection{Poisson's integral for the upper half space}
Let $f :\R^m \to \R$ be a bounded function. We denote {\it Poisson's integral} by
\begin{equation}
Pf (x,h) = \int_{\R^m} p (x-y,h) f(y) dy,\ x \in \R^m ,\ h>0 ,
\end{equation}
where $p$ is {\it Poisson's kernel},
\begin{equation}
p(z,h) = \frac{2}{\sigma_m \( S^m \)} \frac{h}{\( \lvert z\rvert^2 +h^2  \)^{(m+1)/2}},\ z \in \R^m ,\ h>0.
\end{equation}

It is well-known that $Pf$ satisfies the Laplace equation on the upper half space, that is,
\begin{equation}
\( \frac{\pd^2}{\pd h^2} + \De \) Pf(x,h) = \( \frac{\pd^2}{\pd h^2} + \sum_{j=1}^m \frac{\pd^2}{\pd x_j^2} \) Pf(x,h) =0 ,\ x \in \R^m,\ h>0.
\end{equation}
Furthermore, if $f$ is continuous at $x$, then we have
\begin{equation}
\lim_{h \to 0^+} Pf(x ,h) = f(x) .
\end{equation}

\begin{prop}[{\cite[Propositions 5.16 and 5.19]{Skt1}}]\label{existenceP}
Let $f$ be a non-zero non-negative bounded function with compact support.
\begin{enumerate}[$(1)$]
\item For each $h>0$, the function $Pf (\cdot ,h) :\R^m \to \R$ has a maximum value, and it can only be attained at points of the convex hull of $\supp f$.
\item The set of maximum points of $Pf (\cdot ,h)$ converges to the one-point set of the centroid of $f$ as $h$ goes to infinity.
\end{enumerate}
\end{prop}

For a body (the closure of a bounded open set) $\Ome$ in $\R^m$, we use the notation $P_\Ome (x,h) = P \chi_\Ome (x,h)$ for short. Throughout this paper, we call a maximum point of $P_\Ome (\cdot ,h) :\R^m \to \R$ an {\it illuminating center} of $\Ome$ of height $h$.
\subsection{The Cauchy problem for the heat equation}
Let $f :\R^m \to \R$ be a bounded function. Let 
\begin{equation}
Wf (x,t) = \int_{\R^m} w(x-y,t) f(y) dy,\ x \in \R^m ,\ t>0 ,
\end{equation}
where $w$ is {\it Weierstrass' kernel},
\begin{equation}
w (z,t) = \frac{1}{\( 4\pi t \)^{m/2}} \exp \( -\frac{\lvert z \rvert^2}{4t} \) ,\ z \in \R^m ,\ t>0.
\end{equation}

It is well-known that $Wf$ satisfies the heat equation, that is,
\begin{equation}
\( \frac{\pd}{\pd t} - \De \) Wf(x,t) = \( \frac{\pd}{\pd t} - \sum_{j=1}^m \frac{\pd^2}{\pd x_j^2} \) Wf(x,t) =0 ,\ x \in \R^m,\ t>0.
\end{equation}
Furthermore, if $f$ is continuous at $x$, then we have
\begin{equation}
\lim_{t \to 0^+} Wf(x ,t) = f(x) .
\end{equation}

\begin{thm}[{\cite[Theorem 1]{CK}}]\label{existenceW}
Let $f$ be a non-zero non-negative bounded function with compact support.
\begin{enumerate}[$(1)$]
\item For each $t>0$, the function $Wf (\cdot ,t) :\R^m \to \R$ has a maximum value, and it can only be attained at points of the convex hull of $\supp f$.
\item The set of maximum points of $Wf (\cdot ,t)$ converges to the one-point set of the centroid of $f$ as $t$ goes to infinity.
\end{enumerate}
\end{thm}

For a body (the closure of a bounded open set) $\Ome$ in $\R^m$, we use the notation $W_\Ome (x,t) = W \chi_\Ome (x,t)$ for short. Throughout this paper, we call a maximum point of $W_\Ome (\cdot ,t) :\R^m \to \R$ a {\it hot spot} at time $t$.
\subsection{Location of critical points of potentials}

In this subsection, we prepare results on the location of critical points of our potentials $V_\Ome^{(\al)}$, $P_\Ome (\cdot ,h)$ and $W_\Ome (\cdot ,t)$.

Using Alexandrov's reflection principle (\cite{GNN, Ser}), we can restrict the location of critical points of potentials.

\begin{defi}[{\cite[Introduction]{BMS}, \cite[Introduction]{BM}, \cite[Definition 3.3]{O1}}]
{\rm Let $\Ome$ be a body in $\R^m$. Fix $v \in S^{m-1}$ and $b\in \R$. Let 
\[
\Ome^+_{v,b} = \left\{ z \in \R^m \lvert z \cdot v \geq b \right\} \right. ,
\]
and $\Refl_{v,b} :\R^m \to \R^m$ the reflection in the hyperplane $\{ z \in \R^m \vert z \cdot v =b\}$. We denote the {\it maximal folding function} by
\[
l (v) = \inf \left\{ a \in \R \lvert \forall b\geq a,\ \Refl_{v,b} \( \Ome^+_{v,b} \) \subset \Ome \right\} \right.  .
\]
Define the {\it minimal unfolded region} or the {\it heart} of $\Ome$ as
\[
Uf(\Ome ) = \heartsuit (\Ome ) = \bigcap_{v \in S^{m-1}} \left\{ z \in \R^m \lvert z \cdot v \leq l (v) \right\} \right. .
\]
}
\end{defi}

\begin{rem}[{\cite[p.381]{O1}}]
{\rm 
The minimal unfolded region of a body $\Ome$ is contained in the convex hull of $\Ome$ but, in general, not contained in $\Ome$.
}
\end{rem}

The idea of the proof of the following proposition is due to \cite[Introduction]{MS4} or \cite[Theorem 3.5]{O1}.
\begin{prop}\label{critlocation}
Let $\Ome$ be a body in $\R^m$.
\begin{enumerate}[$(1)$]
\item For any $\al >1$, every critical point of $V_\Ome^{(\al)}$ belongs to $Uf (\Ome ) \sm \pd ( \conv \Ome )$.
\item For any $0 \leq \al \leq 1$, every critical point of $V_\Ome^{(\al )}$ belongs to $(\conv \Ome )^\circ \cup \pd \Ome$. Moreover, if $\Ome$ has a piecewise $C^1$ boundary, then every critical point of $V_\Ome^{(\al )}$ belongs to $Uf(\Ome ) \cup \pd \Ome$.
\item For any $\al <0$, every critical point of $V_\Ome^{(\al )}$ belongs to $Uf (\Ome ) \cup \pd \Ome$.
\item For any $h>0$, every critical point of $P_\Ome (\cdot ,h)$ belongs to $Uf (\Ome ) \sm \pd ( \conv \Ome )$.
\item For any $t>0$, every critical point of $W_\Ome (\cdot ,t)$ belongs to $Uf (\Ome ) \sm \pd ( \conv \Ome )$.
\end{enumerate}
\end{prop}

\begin{proof}
(1) Let $x$ be an exterior point of $Uf(\Ome )$. There is a direction $v \in S^{m-1}$ such that $b:= x\cdot v > l(v)$. From the definition of the maximal folding function $l$, $\Refl_{v,b} (\Ome^+_{v,b})$ is contained in $\Ome$. From the first assertion in Remark \ref{diffV}, radial symmetry of the kernel of $V_\Ome^{(\al)}$ implies
\[
\frac{\pd V_\Ome^{(\al )}}{\pd v} (x)=
\begin{cases}
\ds (\al -m) \sign (m-\al ) \int_{\Ome \sm \( \Ome^+_{v,b} \cup \Refl_{v,b} \( \Ome^+_{v,b} \) \)} \lvert x-y \rvert^{\al -m-2} (x-y) \cdot v dy &( \al \neq m),\\
\ds -\int_{\Ome \sm \( \Ome^+_{v,b} \cup \Refl_{v,b} \( \Ome^+_{v,b} \) \)} \frac{(x-y)\cdot v}{\lvert x-y \rvert^2} dy &(\al =m ) .
\end{cases}
\]
Since $(x-y) \cdot v >0$ for any $y$ in the region of the above integral, the derivative of $V_\Ome^{(\al )}$ in $v$ does not vanish. Thus every critical point of $V_\Ome^{(\al)}$ belongs to $Uf(\Ome )$. 

Let $x$ be a point on the boundary of the convex hull of $\Ome$. There is a direction $v \in S^{m-1}$ such that the convex hull of $\Ome$ is contained in the half space $v^\perp_- +x = \{ z+x  \vert  z \cdot v \leq 0 \}$. Since $(x-y) \cdot v >0$ for almost every $y \in \Ome$, we conclude that $x$ is not a critical point of $V_\Ome^{(\al)}$.

(2) Let $x$ be a point not in $(\conv \Ome )^\circ$ or on $\pd \Ome$. We decompose $\R^m$ into
\[
\R^m = \( \conv \Ome \)^c \cup \( \pd \( \conv \Ome \) \sm \pd \Ome \) \cup \( \conv \Ome \)^\circ \cup \pd \Ome .
\]
Using the second assertion in Remark \ref{diffV}, the same argument as in the first assertion implies the non-criticality of $V_\Ome^{(\al )}$ at $x$.

Suppose that $\Ome$ has a piecewise $C^1$ boundary. Let $x$ be a point not in $Uf(\Ome )$ or on $\pd \Ome$. We decompose $\R^m$ into
\[
\R^m = \( Uf(\Ome )^c \cap \c{\Ome} \) \cup \( Uf (\Ome )^c \cap \Ome^c \) \cup Uf (\Ome ) \cup \pd \Ome .
\]
Using the second or fourth assertion in Remark \ref{diffV}, the same argument as in the first assertion implies the conclusion.

(3) Using the second and third assertions in Remark \ref{diffV}, the same argument as in the first assertion implies the conclusion.

(4) and (5) The proofs are same as the first one.  
\end{proof}

\begin{rem}
{\rm 
We refer to \cite{Her2} for the study on the location of a radial center: If $\Ome$ is a smooth convex body in $\R^m$, then the (unique) $r^{1-m}$-center is an interior point of $\Ome$. From Lemma \ref{diffVbd}, Proposition \ref{critlocation} includes her result.
}
\end{rem}

Investigating the concavity of $V_\Ome^{(\al )}$, for the uniqueness of a critical point of $V_\Ome^{(\al )}$, the following results are known:

\begin{lem}[{\cite[Lemma 3.10]{O1}}]\label{O1lem3.10}
Let $\Ome$ be a convex body in $\R^m$. If $\al \leq 1$, then the potential $V_\Ome^{(\al )}$ becomes strictly concave on the interior of $\Ome$. In particular, in this case, $V_\Ome^{(\al )}$ has a unique critical point.
\end{lem}

\begin{lem}[{\cite[Lemma 3.14]{O1}}]\label{O1lem3.14}
Let $\Ome$ be a body in $\R^m$. If $m=1$ and $\al >2$, or if $m \geq 2$ and $\al \geq m+1$, then the potential $V_\Ome^{(\al )}$ becomes strictly concave on $\R^m$. In particular, in this case, $V_\Ome^{(\al )}$ has a unique critical point.
\end{lem} 

Furthermore, the limit point of an $r^{\al -m}$-center was studied.

\begin{thm}[{\cite[Theorem 3.23]{O1}}]\label{O1thm3.23}
Let $\Ome$ be a body in $\R^m$.
\begin{enumerate}[$(1)$]
\item The $r^{\al -m}$-center converges to the circumcenter of $\Ome$ as $\al$ goes to $+\infty$.
\item The limit of any convergent sequence of $r^{\al_j -m}$-centers with $\al_j \to -\infty$ as $j \to +\infty$ is an incenter of $\Ome$.
\end{enumerate}
\end{thm}
\subsection{Power concavity of a function}

In this subsection, we prepare terminologies on the power concavity of a function from \cite{BL, Ken}. The power concavity of a function plays an important role for the study on the uniqueness of a critical point of $P_\Ome (\cdot ,h)$ when $\Ome$ is a convex body.

\begin{defi}[{\cite[Introduction and Property1]{Ken}}]
{\rm
Let $\al \in \R \cup \{ \pm \infty \}$ and $\f$ be a non-negative function defined on $\R^m$. The function $\f$ is {\it $\al$-concave} on $\R^m$ if 
\[
\f \( (1-\la ) x+ \la y\) \geq \( (1-\la ) \f (x)^\al + \la \f (y)^\al \)^{1/\al}
\]
for any $x$, $y \in \R^m$ and $0\leq \la \leq 1$. The function $\f$ is {\it strictly $\al$-concave} on $\R^m$ if the above inequality strictly holds for any distinct $x$, $y \in \R^m$ and $0<\la <1$.

We understand
\[
\( (1-\la ) \f (x)^\al + \la \f (y)^\al \)^{1/\al} =
\begin{cases}
0 &\( \f (x) \f (y) =0 \) ,\\
\max \left\{ \f (x) , \ \f (y) \right\} &( \al =+\infty ),\\
\f (x)^{1-\la} \f (y)^{\la} &( \al =0 ),\\
\min \left\{ \f (x) , \ \f (y) \right\} &( \al =-\infty ) .
\end{cases}
\]
The power concavity for $\al=0$ is also called {\it log-concavity}.
}
\end{defi}

\begin{thm}[{\cite[Theorem 3.3]{BL}}]\label{blthm}
Let $f$ and $g$ be non-negative measurable functions defined on $\R^m$. Suppose that both  $f$ and $g$ have positive $L^1$-norms. Let $\al \geq -1/m$ and $\ga = \al / (1+m\al )$. Then, we have
\[
\int_{\R^m} \sup_{(1-\la )y_1 + \la y_2 =y} \( (1-\la ) f\( y_1 \)^\al + \la g\( y_2 \)^\al \)^{1/\al} dy
\geq \( (1-\la ) \| f \|_1^\ga + \la \| g\|_1^\ga \)^{1/\ga} .
\]
\end{thm}


\begin{prop}[{\cite[Property 4]{Ken}}]\label{Kenprop4}
Let $\al$ be a real constant, and $\f$ a positive $C^2$ function defined on $\R^m$. The function $\f$ is (strictly) $\al$-concave if and only if the inequality
\[
\f (x) \frac{\pd^2 \f}{\pd v^2}(x) +(\al -1) \frac{\pd \f}{\pd v}(x)^2 \leq 0
\]
(resp. strictly) holds for any direction $v \in S^{m-1}$.
\end{prop}

\begin{cor}\label{uniqueness}
Let $\al$ be a real constant, and $\f$ a positive $C^2$ function. If $\f$ is strictly $\al$-concave on $\R^m$, then $\f$ has at most one critical point.
\end{cor}

\begin{proof}
Suppose that $\f$ has two critical points $c_1$ and $c_2$. Thanks to Proposition \ref{Kenprop4}, both $c_1$ and $c_2$ are maximal points of $\f$. Hence, taking $v=(c_1 -c_2 )/\vert c_1 -c_2 \vert$, there exists a $0< \la <1$ such that the point $(1-\la )c_1 + \la c_2$ is a minimal point $\f$ in the direction $v$. But Proposition \ref{Kenprop4} guarantees
\[
0 \leq \frac{\pd^2 \f}{\pd v^2}\((1-\la )c_1 + \la c_2  \) <0 ,
\]
which is a contradiction.
\end{proof}
\section{Balance law}

In this section, we investigate a necessary and sufficient condition for the existence of critical points of  $V_\Ome^{(\al )}$ and $Pf (\cdot ,h)$ independent of $\al$ and $h$, respectively. We show that the condition is the balance law introduced in \cite[Theorem 1]{MS1}.
\subsection{Balance law for Riesz potentials}

Let $\Ome$ be a body (the closure of a bounded open set) in $\R^m$. In this subsection, we discuss the existence of a critical point of $V_\Ome^{(\al )}$ not moving with respect to $\al$. From Example \ref{O1ex3.1}, Lemma \ref{O1lem3.14} and Theorem \ref{O1thm3.23}, an $r^{\al -m}$-center (a critical point) of $V_\Ome^{(\al )}$ moves in general.

\begin{lem}\label{equivalence}
Let $\Ome$ be a body in $\R^m$. The following statements are equivalent:
\begin{itemize}
\item $\Ome$ satisfies the balance law at the origin, that is,
\[
\int_{rS^{m-1} \cap \Ome} v d \sigma (v) =0
\]
for any $r\geq 0$.
\item The complement of $\Ome$ satisfies the balance law at the origin, that is,
\[
\int_{rS^{m-1} \cap \Ome^c} v d \sigma (v) =0
\]
for any $r\geq 0$.
\item For any $\rho \geq 0$, both $\Ome \sm \rho B^m$ and $\Ome \cap \rho B^m$ satisfy the balance law at the origin, that is,
\[
\int_{rS^{m-1} \cap \( \Ome \sm \rho B^m \)} v d \sigma (v) 
= \int_{rS^{m-1} \cap \( \Ome \cap \rho B^m \)} v d \sigma (v)
=0
\]
for any $r\geq 0$.
\end{itemize}
\end{lem}

\begin{proof}
The equivalence of the first and second assertions follows from the identity
\[
0=\int_{r S^{m-1}} v d\sigma (v) 
= \( \int_{rS^{m-1} \cap \Ome} + \int_{rS^{m-1} \cap \Ome^c} \) v d \sigma (v) .
\] 
The equivalence of the first and third assertions follows from the identity
\[
\int_{r S^{m-1} \cap \Ome} v d\sigma (v) 
= \( \int_{rS^{m-1} \cap \( \Ome \sm \rho B^m \)} + \int_{r S^{m-1} \cap \( \Ome \cap \rho B^m \)}  \) v d\sigma (v) ,
\]
which completes the proof.
\end{proof}

\begin{thm}\label{balanceV1}
Let $\Ome$ be a body in $\R^m$. 
\begin{enumerate}[$(1)$]
\item The origin is a critical point of $V_\Ome^{(\al)}$ for any $\al >1$ if and only if the body $\Ome$ satisfies the balance law at the origin.
\item Suppose that the origin is not on the boundary of $\Ome$. The origin is a critical point of $V_\Ome^{(\al)}$ for any $\al \notin [0,1]$ if and only if the body $\Ome$ satisfies the balance law at the origin.
\item Suppose that $\Ome$ has a piecewise $C^1$ boundary, and that the origin is not on the boundary of $\Ome$. The origin is a critical point of $V_\Ome^{(\al)}$ for any $\al \in \R$ if and only if the body $\Ome$ satisfies the balance law at the origin.
\end{enumerate}
\end{thm}

\begin{proof}
(1) The first assertion in Remark \ref{diffV} guarantees that the condition $\nabla V_\Ome^{(\al)}(0)=0$ for any $\al >1$ is equivalent to 
\[
\int_{\R^m} \chi_\Ome (y) \lvert y \rvert^{\al -m-2} y dy =0
\]
for any $\al >1$. Using the polar coordinate, the left hand side becomes
\begin{align*}
\int_0^{+\infty} r^{\al -2} \( \int_{S^{m-1}} \chi_\Ome (rv )v d\sigma (v) \) dr 
&=\int_{-\infty}^{+\infty} e^{-(\al -1)s} \( \int_{S^{m-1}} \chi_\Ome \( e^{-s} v\) vd\sigma (v) \) ds \\
&=\mathcal{B} \left[ \int_{S^{m-1}} \chi_\Ome \( e^{-\bullet} v\) vd\sigma (v) \right] (\al -1) ,
\end{align*}
where we changed the valuable $r$ into $e^{-s}$ in the first equality, and the symbol $\mathcal{B}$ denotes the bilateral Laplace transform. The injectivity of the bilateral Laplace transform 
implies the conclusion.

(2) From the first assertion, it is sufficient to show that $\nabla V_\Ome^{(\al )} (0) =0$ for any $\al <0$ if $\Ome$ satisfies the balance law at the origin. If the origin is in the complement of $\Ome$, then the same argument as in the first assertion implies the conclusion (see also Remark \ref{diffV} (2)). We assume that the origin is in the interior of $\Ome$. 

Thanks to the third assertion in Remark \ref{diffV}, we have
\[
\nabla V_\Ome^{(\al )} (0) = (\al -m) \int_{\Ome^c} \lvert y \rvert^{\al -m-2} y dy .
\]
In the same manner as in the first assertion, the above vector vanishes for any $\al <0$ if the complement of $\Ome$ satisfies the balance law at the origin. Hence Lemma \ref{equivalence} implies the conclusion.

(3) From the second assertion, let us consider the case of $0\leq \al \leq 1$. If the origin is in the complement of $\Ome$, then the same argument as in the first assertion implies the conclusion (see also Remark \ref{diffV} (2)). We assume that the origin is in the interior of $\Ome$.

We take an $0 < \ep < \dist (0, \Ome^c)$. Thanks to Lemma \ref{diffV01}, $\nabla V_\Ome^{(\al )} (0) =0$ for any $0 \leq \al \leq 1$ if and only if 
\[
\int_{\Ome \sm \ep B^m} \lvert y \rvert^{\al -m-2} y dy =0
\]
for any $0\leq \al \leq 1$.
In the same manner as in the first assertion, the above equation holds for any $0\leq \al \leq 1$ if the set $\Ome \sm \ep B^m$ satisfies the balance law at the origin. Hence Lemma \ref{equivalence} implies the conclusion. 
\end{proof}

\begin{cor}\label{determinationV}
Let $\Ome$ be a body in $\R^m$. Suppose that $\Ome$ satisfies the balance law at a point $x$.
\begin{enumerate}[$(1)$]
\item If $m=1$ and $\al \in (2, +\infty ]$, or if $m \geq 2$ and $\al \in [m+1, +\infty ]$, the unique $r^{\al -m}$-center coincides with the point $x$. In particular, the centroid and the circumcenter of $\Ome$ coincide with $x$.
\item If $\Ome$ is convex, then, for any $\al \in (-\infty ,1] \cup [m+1 , +\infty ]$, the unique $r^{\al -m}$-center coincides with the point $x$.
\end{enumerate}
\end{cor}

\begin{proof}
(1) Lemma \ref{O1lem3.14} guarantees the uniqueness of a critical point of $V_\Ome^{(\al)}$. Example \ref{O1ex3.1} and Theorem \ref{O1thm3.23} guarantees that the unique critical point is the centroid and the circumcenter, respectively.

(2) Lemmas \ref{O1lem3.10} and \ref{O1lem3.14} guarantees the uniqueness of a critical point of $V_\Ome^{(\al)}$. Lemma \ref{diffVbd} guarantees that, for any $0 < \al \leq 1$, every $r^{\al -m}$-center of $\Ome$ contained in the interior of $\Ome$. Thus every $r^{\al -m}$-center coincides with the critical point of $V_\Ome^{(\al)}$. 
\end{proof}

\begin{rem}
{\rm 
The same statement as in Corollary \ref{determinationV} holds if the uniqueness of a critical point of $V_\Ome^{(\al )}$. Some sufficient conditions for the uniqueness were studied in, for example, \cite{Her, O3, Skt2}.
}
\end{rem}
\subsection{Balance law for Poisson's integral}

In this subsection, we discuss the existence of a critical point of $Pf (\cdot ,h)$ not moving with respect to $h$.

\begin{lem}\label{kernel}
Poisson's kernel $p(z,h)$ is expressed as
\[
p(z,h)=\frac{2}{\pi^{(m+1)/2}h^m} \int_0^{+\infty} s^m \exp\( - \frac{\lvert z \rvert^2 + h^2}{h^2} s^2\) ds
\]
for any $z \in \R^m$ and $h>0$.
\end{lem} 

\begin{proof}
We remark that the area of $S^m$ is defined as
\[
\sigma_m \( S^m \) 
= \frac{2\pi^{(m+1)/2}}{\Ga \( (m+1)/2\)} 
=2\pi^{(m+1)/2} \( \int_0^{+\infty} e^{-u} u^{(m-1)/2} du \)^{-1}.
\]
Changing the valuable $u$ into 
\[
u \mapsto \frac{\lvert z \rvert^2 +h^2}{h^2}s^2,
\]
direct computation shows the conclusion.
\end{proof}

\begin{rem}
{\rm Let $w$ be Weierstrass' kernel, that is,
\[
w(z,t)= \frac{1}{\( 4\pi t\)^{m/2}} \exp \( -\frac{\lvert z \rvert^2}{4t} \) .
\]
Lemma \ref{kernel} claims that Poisson's kernel is expressed as
\[
p(z,h) = \frac{2}{\sqrt{\pi}} \int_0^{+\infty} e^{-s^2} w\( z, \frac{h^2}{4s^2} \) ds .
\]
}
\end{rem}

\begin{thm}\label{balanceP1}
Let $f$ be a bounded function on $\R^m$. The origin is a critical point of Poisson's integral $Pf(\cdot ,h) :\R^m \to \R$ for any $h>0$ if and only if $f$ satisfies the balance law at the origin, that is,  
\[
\int_{rS^{m-1}} f(v)v  d\sigma (v) =0
\]
for any $r \geq 0$.
\end{thm}

\begin{proof}
Thanks to Lemma \ref{kernel}, the condition $\nabla Pf(0,h) =0$ for any $h>0$ is equivalent to
\[
\int_0^{+\infty} e^{-s^2} s^{m+2} \( \int_{\R^m} \exp \( - \frac{\lvert y \rvert^2}{h^2}s^2 \) f(y) ydy \) ds =0
\]
for any $h>0$. 

Changing the valuable $s$ as $u=(s/h)^2$, the left hand side becomes
\[
\frac{h^{m+3}}{2} \int_0^{+\infty} e^{-h^2u} u^{(m+1)/2} \( \int_{\R^m} e^{-u\lvert y\rvert^2} f(y)ydy \)du
=\frac{h^{m+3}}{2} \mathcal{L} \left[ \bullet^{(m+1)/2} \int_{\R^m} e^{-\bullet \lvert y\rvert^2} f(y) ydy \right] \( h^2\) ,
\]
where the symbol $\mathcal{L}$ denotes the Laplace transform. From the injectivity of the Laplace transform, the condition $\nabla Pf(0,h) =0$ for any $h>0$ is equivalent to 
\[
\int_{\R^m} e^{-u\lvert y \rvert^2}f(y) ydy =0
\]
for any $u > 0$.

Using the polar coordinate, the left hand side becomes
\begin{align*}
\int_0^{+\infty} e^{-ur^2} r^m \( \int_{S^{m-1}} f(rv) v d\sigma (v) \)dr 
&= \frac{1}{2} \int_0^{+\infty} e^{-us} s^{(m-1)/2} \( \int_{S^{m-1}} f\( \sqrt{s}v \)v d\sigma (v) \)ds \\
&=\frac{1}{2} \mathcal{L} \left[ \bullet^{(m-1)/2} \int_{S^{m-1}} f\( \sqrt{\bullet} v\) vd \sigma (v) \right] (u) ,
\end{align*}
where we changed the valuable $r$ as $s=r^2$ in the first equality. The injectivity of the Laplace transform implies the conclusion.
\end{proof}

\begin{cor}\label{determinationP}
Let $f$ be a compactly supported bounded function on $\R^m$.
\begin{enumerate}
\item[$(1)$] If the mass of $f$ does not vanish, then the centroid of $f$ is the unique candidate for a stationary critical point of Poisson's integral $Pf(\cdot ,h):\R^m \to \R$.
\item[$(2)$] If the mass of $f$ vanishes, and if the first moment of $f$ does not vanish, then there is no stationary critical point of Poisson's integral $Pf(\cdot ,h):\R^m \to \R$.
\end{enumerate}
\end{cor}

The proof is due to \cite[p.259]{MS1}.

\begin{proof}
Suppose that $f$ satisfies the balance law at a point $x$, that is,
\[
\int_{S^{m-1}} f(x+rv) v d\sigma (v) =0
\]
for any $r \geq 0$. Then we obtain
\[
0
= \int_0^{+\infty} r^m \( \int_{S^{m-1}} f(x+rv)vd\sigma (v) \)dr 
= \int_{\R^m} f(x+y) y dy 
=\int_{\R^m} f(z) (z-x) dz  ,
\]
which is equivalent to 
\[
\( \int_{\R^m} f(z)dz \) x=\int_{\R^m} f(z)z dz .
\]
Namely, any critical point of $Pf(\cdot ,h)$ independent of $h$ is determined by this equation if it exists.
\end{proof}

\section{Convex polyhedrons having stationary radial centers}

In this section, we characterize a convex polyhedron having stationary radial centers. The results in this section were known as \cite[Theorems 2 and 6]{MS4} in more difficult situation. Let us give elementary proofs to them in our situation.

\begin{prop}\label{characterization}
Let $\Ome$ be a convex polyhedron in $\R^m$, and $r_* = \dist (0, \Ome^c )$. We assume that the intersection $r_* S^{m-1} \cap \pd \Ome$ consists of $k$ points $\{ p_1, \ldots ,p_k \}$. If $\Ome$ satisfies the balance law at the origin, then we have
\[
\sum_{j=1}^{k} p_j =0 .
\]
\end{prop}

\begin{proof}
We remark that, from Corollaries \ref{determinationP} or \ref{determinationV}, the origin is the centroid of $\Ome$. In particular, the origin is in the interior of $\Ome$.

We take a positive $\de$ such that the sphere $(r_* + \de )S^{m-1}$ does not contain any vertex of $\pd \Ome$. Since $\Ome$ satisfies the balance law at the origin, Lemma \ref{equivalence} guarantees that the complement of $\Ome$ satisfies the balance law at the origin. In particular, we have
\[
\int_{\( r_* +\de \) S^{m-1} \cap \Ome^c} v d\sigma (v) 
=0 .
\]
This equation means that the centroid of $(r_* + \de )S^{m-1} \cap \Ome^c$ coincides with the origin. 

On the other hand, the centroid of each connected component of $(r_* +\de )S^{m-1} \cap \Ome^c$ is on the half line from the origin through the corresponding $p_j$. Hence we obtain
\[
\int_{(R+\de )S^{m-1} \cap \Ome^c} v d\sigma (v) = s \sum_{j=1}^k p_j 
\]
for some $s>0$, which completes the proof.
\end{proof}

\begin{rem}
{\rm 
Under the situation in Proposition \ref{characterization}, we have $k \geq 2$, and, for any $1 \leq i \leq k$, there is at least one contact point $p_j$ with $p_i \cdot p_j <0$.
}
\end{rem}

\begin{thm}\label{triangle_quadrangleP}
\begin{enumerate}[$(1)$]
\item Let $\Ome$ be a triangle in $\R^2$. If $\Ome$ satisfies the balance law at the origin, then $\Ome$ is an equilateral triangle centered at the origin.
\item Let $\Ome$ be a quadrangle in $\R^2$. If $\Ome$ satisfies the balance law at the origin, then $\Ome$ is a parallelogram centered at the origin.
\end{enumerate}
\end{thm}

\begin{proof}
(1) From Corollary \ref{determinationV}, the centroid coincides with the circumcenter, which implies the conclusion.

(2) Let $r_*= \dist ( 0, \Ome^c )$. Let us check the following three cases (2.1)--(2.3):

(2.1) Suppose that $r_* S^1 \cap \pd \Ome$ consists of four points $\{ p_1 ,p_2 ,p_3 ,p_4\}$. 

(2.1.1) Suppose $p_2 =-p_1$. From Proposition \ref{characterization}, we have $p_3 = -p_4$, that is, $\Ome$ is a rhombus. 

(2.1.2) Suppose $p_2 \neq -p_1$, that is, $p_1$ and $p_2$ are linearly independent. The point $p_3$ can be expressed as $p_3 =a p_1 + b p_2.$ From Proposition \ref{characterization}, we have $p_4 = -( 1+a )p_1 -(1+b) p_2$. Computing the simultaneous equations
\[
\begin{cases}
r_*^2 = \vert p_3 \vert^2 = \( a^2 +b^2\) r_*^2 + 2 ab p_1 \cdot p_2 ,\\
r_*^2 = \vert p_4 \vert^2 = \( 2+ 2a +2b +a^2 +b^2 \) r_*^2 +2(1+a+b+ab) p_1 \cdot p_2 ,
\end{cases}
\]
we get $(-a)+(-b) =1$. Hence $p_3$ coincides with either $-p_1$ or $-p_2$, and $\Ome$ is a rhombus.

(2.2) Suppose that $r_* S^1 \cap \pd \Ome$ consists of three points $\{ p_1 ,p_2 ,p_3\}$. From Proposition \ref{characterization}, we have
\[
r_*^2 = \lvert p_3 \rvert^2 = \lvert p_1 +p_2 \rvert^2 = 2r_*^2 + 2 p_1 \cdot p_2 ,
\]
which implies
\[
\angle \( p_1 ,p_2 \) = \angle \( p_2 , p_3 \) = \angle \( p_3, p_1 \) = \frac{2\pi}{3} .
\]
Hence the triangle $\triangle p_1 p_2 p_3$ is an equilateral triangle centered at the origin.

Taking a vertex of $\Ome$ with angle $\pi /3$ and folding the domain $\Ome$ in the line joining the origin and the vertex, the origin is not in the minimal unfolded region of $\Ome$, which contradicts to Proposition \ref{critlocation}. Thus this case does not occur.

(2.3) Suppose that $r_* S^1 \cap \pd \Ome$ consists of two points $\{ p_1 ,p_2\}$. From Proposition \ref{characterization}, we have $p_1 =-p_2$, that is, the two edges of $\Ome$ are parallel. We denote by $\{ q_1 , q_2, q_3, q_4\}$ the vertices of $\Ome$. We may assume that
\[
R= \lvert q_1 \rvert = \max \left\{ \lvert q_1 \rvert,\ \lvert q_2 \rvert,\ \lvert q_3 \rvert,\ \lvert q_4 \rvert \right\} 
\]
and $q_j$ is in the $j$-the quadrant. Thanks to Corollary \ref{determinationV}, the origin coincides with the circumcenter of $\Ome$. Thus there is at least one vertex $q_j$ $(j \neq 1)$ with $\vert q_j \vert =R$.

(2.3.1) Suppose $\vert q_2 \vert =R$. Folding the quadrangle $\Ome$ in the orthogonal complement of $p_1$, we conclude that $\vert q_3 \vert = \vert q_4 \vert =R$. Hence $\Ome$ is a rectangle.

(2.3.2) Suppose $\vert q_3 \vert = R$. If $\vert q_2 \vert \neq \vert q_4 \vert$, then the angles $\angle q_4 q_1 q_2$ and $\angle q_2 q_3 q_4$ are different, and we have
\[
\int_{\( R -\de \) S^1 \cap \Ome} v d\sigma (v) \neq 0
\]
for some small $\de>0$. Thus $\vert q_2 \vert = \vert q_4 \vert$, that is, $\Ome$ is a parallelogram.

(2.3.3) Suppose $\vert q_4 \vert =R$. Folding the quadrangle $\Ome$ in the line joining $p_1$ and $p_2$, we conclude that $\vert q_2 \vert = \vert q_3 \vert =R$. Hence $\Ome$ is a rectangle.
\end{proof}

\begin{rem}
{\rm 
The content in this Remark is due to Professor Rolando Magnanini (personal communication to the author). 

There exists a convex body such that it satisfies the balance law and does not have any symmetry. The procedure of the construction of such a convex body is as follows:
\begin{enumerate}[(1)]
\item We start with the unit disc $D^2$ centered at the origin.
\item We choose three directions $u$, $v$ and $w$ such that any non-trivial orthogonal action in $\R^2$ moves the frame $\{ u, v, w \}$, that is, for any $\mathrm{id} \neq g \in O(2)$, $\{ u, v, w\} \neq \{ gu, gv ,gw\}$.
\item For each $r>1$, we add three arcs $C_u (r)$, $C_v(r)$ and $C_w(r)$ whose centroids are on the rays $\overline{0u}$, $\overline{0v}$ and $\overline{0w}$, respectively.
\item We adjust the lengths of $C_u (r)$, $C_v (r)$ and $C_w (r)$ to keep the centroid of $C_u (r) \cup C_v (r) \cup C_w (r)$ at the origin.
\item We can make the union of $D^2$ and $\cup_r (C_u (r) \cup C_v (r) \cup C_w (r))$ convex (see Figure \ref{example}).   
\end{enumerate}
\begin{figure}[htbp]
\begin{center}
\scalebox{1.0}{\includegraphics[clip]{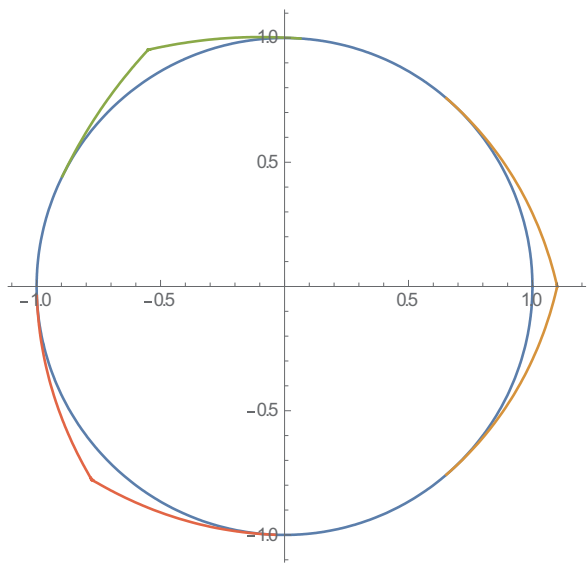}}
\caption{A convex body satisfying the balance law and having no symmetries}
\label{example}
\end{center}
\end{figure}

}
\end{rem}
\section{Strict power concavity of a potential}
In this section, we discuss the strict power concavity of a potential like $P_\Ome (\cdot ,h)$ and $W_\Ome (\cdot ,t)$ when $\Ome$ is a convex body. It derives the uniqueness of a critical point and the smoothness of the locus with respect to the parameter. 

\begin{lem}\label{concavity2m}
Let $\al \in \R \cup \{ \pm \infty \}$ and $\f :\R^m \to [0,+\infty )$ be $\al$-concave on $\R^m$. Then, the function
\[
F: \R^m \times \R^m \ni (x,y) \mapsto \f (x-y ) \in [0,+\infty )
\]
is $\al$-concave on $\R^m \times \R^m$.
\end{lem}

\begin{proof}
Fix $(x_1, y_1 )$, $(x_2 ,y_2) \in \R^m \times \R^m$ and $0 \leq \la \leq 1$. From the $\al$-concavity of $\f$, we have
\begin{align*}
F \( (1-\la ) \( x_1 ,y_1 \) + \la \( x_2 ,y_2 \) \) 
&=\f \( (1-\la ) \( x_1 -y_1 \) + \la \( x_2 -y_2 \) \) \\
&\geq \( (1-\la ) \f \( x_1 -y_1 \)^\al + \la \f \( x_2 -y_2 \)^\al \)^{1/\al} \\
&= \( (1-\la ) F \( x_1 ,y_1 \)^\al + \la F\( x_2, y_2 \)^\al \)^{1/\al},
\end{align*}
which completes the proof.
\end{proof}

\begin{rem}\label{strict_concavity}
{\rm
Let $\al$, $\f$ and $F$ be as in Lemma \ref{concavity2m}. If $\f$ is strictly $\al$-concave on $\R^m$, then, for $(x_1, y_1)$, $(x_2, y_2) \in \R^m \times \R^m$ and $0 \leq \la \leq 1$, the following conditions are equivalent:
\begin{itemize}
\item The equation
\[
F \( (1-\la ) \( x_1 ,y_1 \) + \la \( x_2,  y_2 \) \) = \( (1-\la ) F \( x_1 ,y_1 \)^\al + \la F \( x_2 ,y_2 \)^\al \)^{1/\al}
\]
holds.
\item Any of the conditions $x_1 -y_1 = x_2 -y_2$, $\la =0$ or $\la =1$ holds.
\end{itemize}
}
\end{rem}

\begin{thm}\label{power-concavity}
Let $\Ome$ be a convex body in $\R^m$, $\al \geq -1/m$ and $\ga = \al /(1+m\al )$. Let $\f :\R^m \to [0,+\infty )$ be a continuous strictly $\al$-concave function. Then, the function
\[
G_\Ome (x) =  \int_\Ome \f (x-y)  dy ,\ x \in \R^m ,
\]
is strictly $\ga$-concave.
\end{thm}

\begin{proof}
Fix distinct two points $x_1$ and $x_2 \in \R^m$. Let $0<\la <1$. Put
\[
F(x,y) =\f (x-y) ,\ F_\Ome (x,y) = F(x,y) \chi_\Ome (y) ,\ (x,y) \in \R^m \times \R^m .
\] 

We remark that, for any exterior point $y$ of $\Ome$, the convexity of $\Ome$ implies
\[
\sup_{(1-\la )y_1 + \la y_2 =y} \( (1-\la ) F_\Ome \( x_1 ,y_1 \)^\al + \la F_\Ome \( x_2 ,y_2 \)^\al \)^{1/\al} =0 .
\]

Let $\Ome'= ( \Ome -\la (x_1-x_2) ) \cap (\Ome + (1-\la) (x_1 -x_2))$. Direct computation shows that $y$ is in $\Ome'$ if and only if there exits two points $y_1$ and $y_2$ with
\[
\begin{cases}
(1-\la )y_1 +\la y_2 =y,\\
x_1 -y_1 =x_2 -y_2,\\
y_1, \ y_2 \in \Ome.
\end{cases}
\]
Therefore, for any point $y \in \Ome$, Remark \ref{strict_concavity} implies
\[
\sup_{(1-\la )y_1 + \la y_2 =y} \( (1-\la ) F_\Ome \( x_1 ,y_1 \)^\al + \la F_\Ome \( x_2 ,y_2 \) \)^\al 
\begin{cases}
= F \( (1-\la ) x_1+ \la x_2,  y \) &\( y \in \Ome' \) , \\
< F \( (1-\la ) x_1 + \la x_2,  y  \) &\( y \notin \Ome' \) .
\end{cases}
\]
Here, the inequality follows from the compactness of $\Ome$ and the continuity of $\f$.

Since the set $\Ome \sm \Ome'$ has a positive measure, Theorem \ref{blthm} implies
\begin{align*}
G_\Ome \( (1-\la ) x_1 + \la x_2 \) 
&= \( \int_{\Ome \sm \Ome'} + \int_{\Ome'} \) F \( (1-\la ) x_1 + \la x_2 , y  \) dy \\
&> \( \int_{\Ome \sm \Ome'} + \int_{\Ome'} \) \sup_{(1-\la )y_1 + \la y_2 =y} \( (1-\la )F_\Ome \( x_1 ,y_1 \)^\al + \la F_\Ome \( x_2,  y_2 \)^\al \)^{1/\al} dy \\
&= \int_{\R^m} \sup_{(1-\la )y_1 + \la y_2 =y} \( (1-\la )F_\Ome \( x_1 ,y_1 \)^\al + \la F_\Ome \( x_2,  y_2 \)^\al \)^{1/\al} dy \\
&\geq \( (1-\la ) G_\Ome \( x_1 \)^\ga + \la G_\Ome \( x_2 \)^\ga \)^{1/\ga} ,
\end{align*}
which completes the proof.
\end{proof}

\begin{rem}
{\rm 
We refer to \cite[Corollary 3.5]{BL} for the power concavity of a function of the form in Theorem \ref{power-concavity}. We remark that the {\em strict} power concavity was not discussed in \cite{BL}.
}
\end{rem}

\begin{lem}\label{concavitym}
Let $\psi :[0,+\infty ) \to \R$ be decreasing and strictly concave (increasing and strictly convex) on $[0,+\infty)$. Then, the function
\[
\R^m \ni z \mapsto \psi \( \lvert z \rvert \) \in \R
\]
is strictly concave (resp. strictly convex) on $\R^m$.
\end{lem}

\begin{proof}
Fix distinct two points $x$ and $y \in \R^m$. Let $0<\la <1$. We have
\[
\psi \( \lvert (1-\la ) x+ \la y \rvert \) 
\geq \psi \( (1-\la ) \lvert x \rvert + \la \lvert y \rvert \) 
> (1-\la ) \psi \( \lvert x \rvert \) + \la \psi \( \lvert y \rvert \) .
\]
Here, the first inequality follows from the triangle inequality and the decreasing behavior of $\psi$, and the second inequality is a direct consequence of the strict concavity of $\psi$.
\end{proof}



\begin{cor}\label{concavityWP}
Let $\Ome$ be a convex body in $\R^m$. 
\begin{enumerate}[$(1)$]
\item For each positive $h$, the function $P_\Ome (\cdot ,h) :\R^m \to \R$ is strictly $-1$-concave on $\R^m$.
\item For each positive $t$, the function $W_\Ome (\cdot ,t): \R^m \to \R$ is strictly log-concave on $\R^m$.
\end{enumerate}
\end{cor}

\begin{proof}
From Lemma \ref{concavitym}, Poisson's kernel $p(\cdot ,h)$ is strictly $-1/(m+1)$-concave on $\R^m$. From Theorem \ref{power-concavity}, we obtain the conclusion.
\end{proof}

\begin{cor}\label{uniquenessWP}
Let $\Ome$ be a convex body in $\R^m$. 
\begin{enumerate}[$(1)$]
\item For each positive $h$, the function $P_\Ome (\cdot ,h) :\R^m \to \R$ has a unique critical point. In particular, every convex body has a unique illuminating center for each $h>0$.
\item For each positive $t$, the function $W_\Ome (\cdot ,t) :\R^m \to \R$ has a unique critical point. In particular, every convex body has a unique hot spot for each $t>0$.
\end{enumerate}
\end{cor}

\begin{proof}
From Corollaries \ref{uniqueness} and \ref{concavityWP}, the function $P_\Ome (\cdot ,h)$ has at most one critical point. From Proposition \ref{existenceP}, $P_\Ome (\cdot ,h)$ has at least one maximum (critical) point. 
\end{proof}

\begin{cor}\label{smoothnessWP}
Let $\Ome$ be a convex body in $\R^m$. 
\begin{enumerate}[$(1)$]
\item The map that assigns the unique critical point of $P_\Ome (\cdot ,h)$ for each $h>0$ is smooth on the interval $(0,+\infty)$.
\item The map that assigns the unique critical point of $W_\Ome (\cdot ,t)$ for each $t>0$ is smooth on the interval $(0,+\infty)$.
\end{enumerate}
\end{cor}

\begin{proof}
Let $c(h)$ be the unique critical point of $P_\Ome (\cdot ,h)$ for each $h>0$. From Proposition \ref{Kenprop4} and Corollary \ref{concavityWP}, at $c(h)$, the Hessian of $P_\Ome (\cdot , h)$ does not vanish. Hence the implicit function theorem implies the smoothness of the map $h \mapsto c(h)$.
\end{proof}

\begin{rem}
{\rm 
For the function $W_\Ome (\cdot ,t)$, the statements in Corollaries \ref{concavityWP}, \ref{uniquenessWP} and \ref{smoothnessWP} were essentially known in \cite[Section 6]{BL} (see also \cite[p.2]{MS4}).
}
\end{rem}
\section{Appendix}

In this section, for the Laplace equation on the upper half space, we give corresponding results to \cite[Theorem 2]{MS1} and \cite[Theorem 4]{MS2}. The proofs are similar to those in \cite{MS1, MS2} (see also \cite[Theorem 1]{Skg}).

\begin{thm}\label{balanceP2}
Let $\Ome$ be a domain in $\R^m$ containing the origin, $r_* = \dist (0, \Ome^c )$, and $(a,b)$ a non-empty interval in $(0,+\infty)$. Let $u :\Ome \times (a,b) \to \R$ satisfy the Laplace equation in the cylinder $\Ome \times (a,b )$. The origin is a critical point of $u(\cdot ,h) :\Ome \to \R$ for any $h \in (a,b)$ if and only if $u$ satisfies the balance law at the origin, that is, 
\[
\int_{rS^{m-1}} u(v,h) vd \sigma (v) =0
\]
for any $r \in [0,r_*)$ and $h \in (a,b)$.
\end{thm}

\begin{proof}
We first show the ``if'' part. Suppose that $u$ satisfies the balance law at the origin. The divergence theorem implies 
\[
\int_{rB^m} \nabla u(x,h) dx = \frac{1}{r} \int_{rS^{m-1}} u(x,h) x d\sigma (x) =0 .
\]
Hence we obtain
\[
\nabla u(0,h) = \lim_{r\to 0^+} \frac{1}{\Vol \( rB^m \)} \int_{rB^m} \nabla u(x,h) dx =0 
\]
for any $h \in (a,b)$.

Next, we show the ``only if'' part. Suppose that $\nabla u(0,h) =0$ for any $h \in (a,b)$. Put
\[
F(r,h) = \int_{S^{m-1}} u(rv, h) v d\sigma (v) ,\ G(r,h) = \int_0^r F( \rho ,h) d\rho .
\]
Since $u$ satisfies the Laplace equation in the cylinder $\Ome \times (a,b)$, we have
\begin{align*}
0&= \int_{S^{m-1}} \( \frac{\pd^2}{\pd h^2} + \De \) u(rv,h) v d\sigma (v) \\
&=\int_{S^{m-1}} \( \(\frac{\pd^2}{\pd h^2} +\frac{\pd^2}{\pd r^2} +\frac{m-1}{r} \frac{\pd}{\pd r} +\frac{1}{r^2} \De_{S^{m-1}} \) u(rv,h) \) v d\sigma (v) \\
&= \(\frac{\pd^2}{\pd h^2} +\frac{\pd^2}{\pd r^2} +\frac{m-1}{r} \frac{\pd}{\pd r}  \) F(r,h) +\frac{1}{r^2} \int_{S^{m-1}} \( \De_{S^{m-1}} u(rv,h) \)v d\sigma (v) ,
\end{align*}
where $\De_{S^{m-1}}$ denotes the Laplace-Beltrami operator on $S^{m-1}$. It is well-known that $\De_{S^{m-1}} v = -(m-1)v$ (see \cite[pp.34--35]{C}). Therefore, integration by parts implies
\[
\int_{S^{m-1}} \( \De_{S^{m-1}} u(rv,h) \)v d\sigma (v) = -(m-1) \int_{S^{m-1}} u(rv,h)v d\sigma (v) = -(m-1) F(r,h) ,
\]
and we have
\[
\(\frac{\pd^2}{\pd h^2} +\frac{\pd^2}{\pd r^2} +\frac{m-1}{r} \frac{\pd}{\pd r}  -\frac{m-1}{r^2}  \) F(r,h) =0
\]
for any $r \in [0,r_*)$ and $h \in (a,b)$.

Let us show 
\[
\frac{\pd^n F}{\pd r^n}(0,h) =0
\]
for any non-negative integer $n$ by induction. From the definition of $F$, we have
\[
F(0,h) = u(0,h) \int_{S^{m-1}} v d\sigma (v) =0,
\]
and, from the assumption of $u$, we have
\[
\frac{\pd F}{\pd r}(0,h) = \int_{S^{m-1}} \( \nabla u(0,h) \cdot v \) v d\sigma (v) =0 .
\]
Suppose 
\[
\frac{\pd F}{\pd r}(0,h) = \cdots = \frac{\pd^n F}{\pd r^n}(0,h) =0 .
\]
Using the equation for $F$, we have
\begin{align*}
0&= \frac{\pd^{n+1}}{\pd r^{n+1}} \(r^2 \frac{\pd^2}{\pd h^2} +r^2 \frac{\pd^2}{\pd r^2} +(m-1) r \frac{\pd}{\pd r}  -(m-1)  \) F(r,h)\\
&=\frac{\pd^2}{\pd h^2} \sum_{j=0}^{n+1} \binom{n+1}{j} \(\frac{\pd^j}{\pd r^j} r^2 \) \frac{\pd^{n+1-j}F}{\pd r^{n+1-j}} (r,h)
+ \sum_{j=0}^{n+1} \binom{n+1}{j} \(\frac{\pd^j}{\pd r^j} r^2 \) \frac{\pd^{n+3-j}F}{\pd r^{n+3-j}} (r,h) \\
&\quad + (m-1) \sum_{j=0}^{n+1} \binom{n+1}{j} \(\frac{\pd^j}{\pd r^j} r \) \frac{\pd^{n+2-j}F}{\pd r^{n+2-j}} (r,h)
-(m-1) \frac{\pd^{n+1}F}{\pd r^{n+1}} (r,h) .
\end{align*}
Putting $r=0$, the assumption of induction implies
\[
n(m+n)  \frac{\pd^{n+1} F}{\pd r^{n+1}} (0,h) =0 ,
\]
that is, for any non-negative integer $n$, the $n$-th derivative of $F$ with respect to $r$ vanishes at $(0,h)$.

Using the properties of $F$, let us show that $G=0$ in $[0,r_*) \times (a,b)$. Integrating the equation for $F$ with respect to $r$, and using the relation $\pd G /\pd r=F$, we get
\[
\( \frac{\pd^2}{\pd h^2} +\frac{\pd^2}{\pd r^2} +\frac{m-1}{r} \frac{\pd}{\pd r} \) G(r,h)=0.
\]
Since all the derivatives of $F$ with respect to $r$ vanish at $(0,h)$, those of $G$ so are. Furthermore, since $G(0,h)=0$ for any $h \in (a,b)$, all the derivatives of $G$ with respect to $h$ vanish. Let $\tilde{G}(x,h) = G( \vert x \vert ,h)$. Then $\tilde{G}$ is harmonic in $r_* B^m \times (a,b)$. Expanding the function $\tilde{G}$ in a neighborhood of $(0, (a+b)/2)$, we obtain
\[
\tilde{G}(x,h) = \sum_{n=0}^{+\infty} \sum_{\lvert \al \rvert =n} \frac{1}{\al !} D^\al \tilde{G} \( 0,\frac{a+b}{2} \) \( x, h-\frac{a+b}{2} \)^\al =0.
\]
The analyticity of $\tilde{G}$ implies $G=0$ in $[0,r_*) \times (a,b)$, and hence $F=0$ in $[0,r_*) \times (a,b)$.
\end{proof}

\begin{thm}\label{balanceP3}
Let $f$ be a bounded function on $\R^m$. The origin is a zero point of Poisson's integral $Pf(\cdot ,h) :\R^m \to \R$ for any $h>0$ if and only if the function $f$ satisfies the balance law, that is,
\[
\int_{rS^{m-1}} f(v) d\sigma (v) =0
\]
for any $r\geq 0$.
\end{thm}

\begin{proof}
Thanks to Lemma \ref{kernel}, the condition $Pf(0,h)=0$ for any $h>0$ is equivalent to
\[
\int_0^{+\infty} e^{-s^2} s^m \( \int_{\R^m} \exp \( -\frac{\lvert y\rvert^2}{h^2}s^2 \) f(y)dy \)ds =0 
\]
for any $h>0$. 

Changing the valuable $s$ as $u= (s/h)^2$, the left hand side becomes
\[
\frac{h^{m+1}}{2} \int_0^{+\infty} e^{-h^2u} u^{(m-1)/2} \( \int_{\R^m} e^{-u \lvert y \rvert^2} f(y) dy \) du
= \frac{h^{m+1}}{2} \mathcal{L} \left[ \bullet^{(m-1)/2} \int_{\R^m} e^{-\bullet \lvert y \rvert^2} f(y) dy \right] \( h^2 \) ,
\]
where the symbol $\mathcal{L}$ denotes the Laplace transform. Therefore, the condition $Pf (0,h) =0$ for any $h>0$ is equivalent to 
\[
\int_{\R^m} e^{-u \lvert y \rvert^2} f(y) dy =0
\]
for any $u> 0$.

Using the polar coordinate, the left hand side becomes
\begin{align*}
\int_0^{+\infty} e^{-ur^2} r^{m-1} \( \int_{S^{m-1}} f(rv) d\sigma (v) \) dr
&=\frac{1}{2} \int_0^{+\infty} e^{-us} s^{(m-2)/2} \( \int_{S^{m-1}} f \( \sqrt{s} v\)  d\sigma (v) \) ds \\
&=\frac{1}{2} \mathcal{L} \left[ \bullet^{(m-2)/2} \int_{S^{m-1}} f\( \sqrt{\bullet} v\)  d\sigma (v) \right] (u) ,
\end{align*}
where we changed the valuable $r$ as $s=r^2$ in the first equality. Hence the injectivity of the Laplace transform implies the conclusion.
\end{proof}

\begin{thm}\label{balanceP4}
Let $\Ome$ be a domain in $\R^m$ containing the origin, $r_* = \dist (0, \Ome^c )$, and $(a,b)$ a non-empty interval in $(0,+\infty)$. Let $u :\Ome \times (a,b) \to \R$ satisfy the Laplace equation in the cylinder $\Ome \times (0,+\infty )$. The origin is a zero point of $u(\cdot ,h) :\Ome \to \R$ for any $h \in (a,b)$ if and only if $u$ satisfies the balance law, that is, 
\[
\int_{S^{m-1}} u(rv,h) d \sigma (v) =0
\]
for any $r \in [0,r_*)$ and $h \in (a,b)$.
\end{thm}

\begin{proof}
We immediately show the ``if'' part as
\[
u(0,h) = \frac{1}{\sigma \( S^{m-1} \)} \int_{S^{m-1}} u(0,h) d\sigma (v) =0 .
\]

Let us show the ``only if'' part. Let 
\[
F(r,h) = \int_{S^{m-1}} u(rv,h) d\sigma (v)  
\]
for $r \in [0,r_*)$ and $h \in (a,b)$. Since the function $u$ satisfies the Laplace equation in the cylinder $\Ome \times (a,b)$, we have
\begin{align*}
0&= \int_{S^{m-1}} \( \frac{\pd^2}{\pd h^2} + \De \) u(rv,h) d\sigma (v) \\
&=\int_{S^{m-1}} \( \frac{\pd^2}{\pd h^2} +\frac{\pd^2}{\pd r^2} +\frac{m-1}{r} \frac{\pd}{\pd r} + \frac{1}{r^2}  \De_{S^{m-1}} \) u(rv,h) d\sigma (v) \\
&=\( \frac{\pd^2}{\pd h^2} +\frac{\pd^2}{\pd r^2} +\frac{m-1}{r} \frac{\pd}{\pd r} \) F(r,h) .
\end{align*}

By induction, let us show 
\[
\frac{\pd^n F}{\pd r^n}(0,h) =0
\]
for any non-negative integer $n$. We remark that the assumption of $u$ implies $F(0,h) = 0$. Suppose 
\[
\frac{\pd F}{\pd r}(0,h) = \cdots = \frac{\pd^n F}{\pd r^n}(0,h) =0.
\]
From the equation for $F$, we have
\begin{align*}
0&= \frac{\pd^n}{\pd r^n} \( r\frac{\pd^2}{\pd h^2} +r\frac{\pd^2}{\pd r^2} +(m-1) \frac{\pd}{\pd r} \) F(r,h) \\
&=\frac{\pd^2}{\pd h^2} \sum_{j=0}^n \binom{n}{j} \( \frac{\pd^j}{\pd r^j}r\) \frac{\pd^{n-j} F}{\pd r^{n-j}} (r,h)  +\sum_{j=0}^n \binom{n}{j} \( \frac{\pd^j}{\pd r^j}r\) \frac{\pd^{n+2-j} F}{\pd r^{n+2-j}} (r,h) +(m-1)  \frac{\pd^{n+1} F}{\pd r^{n+1}}(r,h) .
\end{align*}
Putting $r=0$, the assumption of induction implies 
\[
(n+m-1) \frac{\pd^{n+1} F}{\pd r^{n+1}}(0,h) =0 .
\]
Hence all the derivatives of $F(r,h)$ with respect to $r$ vanish at $(0,h)$. Furthermore, since $F(0,h) =0$ for any $h$, all the derivatives of $F(0,h)$ with respect to $h$ vanish.

Let $\tilde{F}(x,h) =F(\vert x \vert ,h)$. Then $\tilde{F}$ is harmonic in $r_* B^m \times (a,b)$. Using the properties on the derivatives of $F$, in a neighborhood of $(0,(a+b)/2)$, we have
\[
\tilde{F} (x,h) = \sum_{n=0}^{+\infty} \sum_{\lvert \al \rvert =n} \frac{1}{\al !} D^\al \tilde{F} \( 0, \frac{a+b}{2} \) \( x,h-\frac{a+b}{2} \)^\al =0 .
\]
The analyticity of $\tilde{F}$ guarantees $\tilde{F}=0$ in $r_* B^m \times (a,b)$, and hence $F=0$ in $[0,r_*) \times (a,b)$.
\end{proof}

\no
Faculty of Education and Culture,\\
University of Miyazaki,\\
1-1, Gakuen Kibanadai West, Miyazaki city, Miyazaki prefecture, 889-2155, Japan\\
E-mail: sakata@cc.miyazaki-u.ac.jp

\end{document}